\documentclass[preprint,12pt]{elsarticle}

\usepackage{symbols}
\usepackage{amsfonts}
\usepackage{graphicx}
\usepackage{amsmath}
\usepackage{amssymb}
\usepackage{amsthm}

\usepackage{subfig}	
\usepackage{tikz}	
\usepackage{amsbsy} 
\usepackage{bbold}
\usepackage{amsbsy,amsfonts} 
\usepackage{booktabs}
\usepackage[]{algorithm2e}
\usepackage{pgf,tikz,pgfplots}
\pgfplotsset{compat=newest} 
\usepackage{subfig}
\usetikzlibrary{arrows}
\usepackage{float,lscape}
\usepackage{url}

\definecolor{ffffff}{rgb}{0.9,0.9,0.9}

\newtheorem{theorem}{Theorem}
\newtheorem{proposition}[theorem]{Proposition}
\newtheorem{lemma}[theorem]{Lemma}
\newdefinition{remark}{Remark}
\newproof{pf}{Proof}
\newproof{pot}{Proof of Theorem \ref{thm2}}
\theoremstyle{definition} 
\newtheorem{definition}{Definition}

\newcommand{\R}{\mathbb{R}}
\newcommand{\oR}{\overline{\R}}
\newcommand{\oo}{o}
\newcommand{\OO}{{\cal O}}
\newcommand{\mv}{\,\big |\,}
\newcommand{\bmv}{\,\Big |\,}
\newcommand{\skalp}[1]{\langle #1\rangle}
\newcommand{\rge}{{\rm rge\;}}
\newcommand{\Sp}{{\cal S}}
\newcommand{\Z}{{\cal Z}}
\newcommand{\M}{{\cal M}}
\newcommand{\B}{{\cal B}}

\newcommand{\longsetto}[1]{\mathop{\longrightarrow}\limits^{#1}}
\newcommand{\SCD}{SCD\ }
\newcommand{\ssstar}{semismooth$^*$ }
\newcommand{\ZmP}{\Z_m^{P,W}}
\newcommand{\onabla}{\overline\nabla}
\newcommand{\xb}{\bar x}
\newcommand{\yb}{\bar y}
\newcommand{\ee}[2]{{#1}^{(#2)}}
\newcommand{\argmin}{\mathop{\rm arg\,min\,}}
\newcommand{\dist}[1]{{\rm dist}(#1)}
\newcommand{\norm}[1]{\|#1\|}
\newcommand{\inn}{{\rm int\,}}
\renewcommand{\tto}{\rightrightarrows}

\begin{document}

\begin{frontmatter}

\title{On the application of the SCD semismooth* Newton method to solving Stokes problem with stick-slip boundary conditions}

\author[1]{V.\ Arzt}
\ead{vladimir.arzt@vsb.cz}
  
\author[1]{P.\ Beremlijski}
\ead{petr.beremlijski@vsb.cz}
  
\author[2,3]{H.\ Gfrerer}
\ead{helmut.gfrerer@ricam.oeaw.ac.at}

\author[3]{J.V.\ Outrata}
\ead{outrata@utia.cas.cz}

\address[1]{Faculty of Electrical Engineering, VSB-Technical University of Ostrava,\\ 17. listopadu 15, 708\,33 Ostrava-Poruba, Czech Republic}

\address[2]{Johann Radon Institute for Computational and Applied Mathematics (RICAM), Austrian Academy of Sciences,\\
Altenbergerstr. 69, 4040 Linz, Austria}

\address[3]{Institute of Information Theory and Automation, Czech Academy of Sciences,\\ Pod vod\'{a}renskou v\v{e}\v{z}\'{\i}~4, 182\,08 Praha~8, 
Czech Republic}

\begin{abstract}
The paper deals with the 3D Stokes problem with Navier-Tresca stick-slip boundary conditions. A weak formulation of this problem leads to a variational inequality of the second kind, coupled with an equality constraint. This problem is then approximated using the mixed finite element method, yielding a generalized equation, to the numerical solution of which we implement a variant of the SCD semismooth* Newton method. This includes also a globalization technique ensuring convergence for arbitrary starting points. Numerical experiments demonstrate the effeciency of this approach.
\end{abstract}

\begin{keyword}
Stokes problem \sep stick-slip boundary condition \sep Navier-Tresca conditions \sep SCD mapping \sep SCD semismooth* Newton method 
\end{keyword}

\end{frontmatter}

\begin{keyword}
Stokes problem \sep stick-slip boundary condition \sep Navier-Tresca conditions \sep SCD mapping \sep SCD semismooth* Newton method 
\end{keyword}

\section{Introduction}

To model the flow of incompressible fluids, one often makes use of the Stokes equations, which are a simplification of the Navier-Stokes equations. It accurately characterizes low-speed, non-turbulent flows, i.e., flows with low Reynolds numbers.

The standard kinematic boundary condition for flow problems is represented by a no-slip condition. This condition states that the flow velocity at the~boundary created by a solid, impermeable wall is zero. However, in many real-world scenarios, a non-zero flow velocity can be observed at the~boundary (see~\cite{rothstein} and~\cite{fujita}). To address these situations, boundary conditions that allow for slip at the boundary have been introduced. At this boundary, there may be nodes where sticking occurs (i.e., they are in a stick state) and nodes where slipping occurs (i.e., they are in a slip state). These so-called stick-slip boundary conditions can be described either by a smooth function (the Navier boundary condition, see~\cite{navier}) or by a nonsmooth function (the~Tresca and Navier-Tresca boundary conditions, see~\cite{roux,roux-tani,haslinger-stebel-sassi,haslinger-makinen-stebel}). The~Tresca and Navier-Tresca conditions are more efficient in describing problems where slip has a threshold character. When the Stokes problem is formulated with the Tresca or Navier-Tresca condition in a weak formulation, it leads to an~inequality-type problem. It is known that a solution to the Stokes problem with either of these boundary conditions exists and is unique (see~\cite{haslinger-makinen-stebel}).

The Stokes problem with stick-slip boundary conditions is typically addressed by transforming the respective nonsmooth variational inequality into a dual problem. This dual problem is solved using methods such as the active set method, the path-following interior point method, or the semismooth Newton method. These approaches leverage Lagrange multipliers to enforce constraints including impermeability and incompressibility, and employ the~preconditioned conjugate gradient method for linear systems, often with specialized diagonal preconditioners (see~\cite{kucera-haslinger-satek-jarosova, haslinger-kucera-sassi-satek}). Each method ultimately solves a~system of equations derived from smoothing the original system. For numerical approximation, P1-bubble/P1 elements are frequently used (see~\cite{arnold-brezzi-fortin,kucera-arzt-koko}). Their advantage is a lower degree of freedom than, for example, P2-P1 elements. 

In this paper we employ to the numerical solution of the Stokes problem with the Navier-Tresca stick-slip boundary conditions a different approach, based on the recently developed SCD (Subspace Containing Derivative) semismooth* Newton method. In the sequel, this method will be denoted by the acronym SSSN. To achieve this goal, we will proceed in a similar way as employed in the papers~\cite{gfrerer-outrata-valdman} and~\cite{gfrerer-mandlmayr-outrata-valdman} in case of different variational problems. Note that, unlike the methods mentioned in the previous paragraph, this technique does not require any substantial reformulation/smoothing of the original nonsmooth problem. The SSSN method is a~powerful numerical tool for solving a class of generalized equations (systems with both single- and multi-valued parts), including various types of variational inequalities. It has been developed by Gfrerer and Outrata in \cite{GfrOut22} and it improves an earlier version of it (\cite{gfrerer-outrata}) by using the SCD property of the multi-valued part of the respective generalized equation (GE). Thanks to this modification, one can  construct a suitable approximation (linearization) of the considered GE in an efficient way. Similarly to the classical semismooth Newton methods, the SSSN method achieves locally superlinear convergence.  
To use the described method, it is necessary just to rewrite the~Stokes problem with stick-slip boundary conditions in form of a GE.
\if{
for which it is necessary to use generalized differential calculus by B. Mordukhovich (see~\cite{mordukhovich1,mordukhovich2}). This method has already been successfully used to solve the contact problem with Coulomb friction (see~\cite{gfrerer-mandlmayr-outrata-valdman}).
}\fi

The outline of the paper is as follows. Section 2 is devoted to the Stokes problem with the Navier-Tresca stick-slip boundary condition. In its first part, we focus on the continuous formulation of this problem. Next, we present a discrete formulation based on the mixed finite element method and, finally, optimality conditions are derived. In Section 3, we describe our main tool, the SSSN method.
\if{
We present this method for solving a generalization of the algebraic formulation of the Stokes problem presented in the previous section.
}\fi
Among other things, this requires to introduce the so-called SCD mappings and some notions from the respective theory. Thereafter, we present the two fundamental steps of the SSSN method: The~approximation step and the Newton step. Section 4 focuses on applying the~SSSN method to the considered problem and includes all implementation details. Section 5 is devoted to test examples showing the~efficiency of the~presented approach.

Our notation is standard:
Let $A$ be a set in $\dabR^s$, $\overline{x} \in A$ and $A$ be locally closed around $\overline{x}$. Then the tangent cone to $A$ at $\overline{x}$ is given by
$$T_A\left(\overline{x}\right) = \Limsup_{t \downarrow 0} \frac{A-\overline{x}}{t},$$
where "$\Limsup$" stands for the Painlev\'e-Kuratowski outer set limit (see~\cite{RoWe98}).
$||\cdot||$ denotes the Euclidean norm in 3D.

For two vectors $u,v\in\dabR^d$ we use the notation $\skalp{u,v} := u_1\cdot v_1 + \ldots + u_d\cdot v_d$. For two matrices $A,B\in\mathbb{R}^{d\times d}$ we denote by 
$A:B := \sum_{i,j=1}^d A_{ij}B_{ij}$ the~Frobenius inner product (double contraction).
Given  a real-valued convex function $f$, $\partial f(x)$ stands for the classical Moreau-Rockafellar subdifferential of $f$ at $x$. By $x \longsetto{A} \bar x$ we denote convergence within a set $A$. By $I_n$ we denote the unit matrix of the dimension $n$.

\section{Stokes problem}
\label{state_problem}
In this section, the classical formulation of the Stokes problem with the~Navier-Tresca stick-slip boundary condition is presented. Thereafter, the~weak formulation and the respective discretized problem are derived. Finally, optimality conditions are stated.

\subsection{Continuous and two-field weak formulation of the problem}
Consider a bounded domain \(\Omega \subset \dabR^3\) with a sufficiently smooth boundary \(\partial \Omega\). The boundary is divided into three non-empty disjoint parts such that $\partial\Omega=\overline{\gamma}_D\cup\overline{\gamma}_N\cup\overline{\gamma}_S$. On this domain, the steady flow of a viscous incompressible Newtonian fluid is modeled by the Stokes system with the Dirichlet and the Neumann boundary conditions on $\gamma_D$, $\gamma_N$, respectively, and with the~impermeability and nonlinear (Navier-Tresca) stick-slip boundary conditions on $\gamma_S$.
The classical formulation reads as follows:\\ Find a velocity field $\bu:\overline{\Omega}\rightarrow\dabR^3$ and a pressure field $p:\overline{\Omega}\rightarrow \dabR$ such that
\begin{equation} \label{eq1:1}
\begin{array}{rcll}
-2\nu\nabla\cdot\nabla_S\, \bu+\nabla p & = & \bff & \mbox{in }\Omega,\\
\nabla\cdot\bu & = & 0 & \mbox{in }\Omega,\\
\bu & = & \bu_D & \mbox{on }\gamma_D,\\
\bsigma & = & \bsigma_N & \mbox{on }\gamma_N,\\
u_n & = & 0 & \mbox{on }\gamma_S,\\
||\bsigma_t+\kappa\bu_t|| & \leq & g & \mbox{on }\gamma_S,\\
\bsigma_t\cdot\bu_t+g||\bu_t||+\kappa\bu_t\cdot\bu_t & = & 0 & \mbox{on }\gamma_S,
\end{array}
\end{equation}
where $\nu>0$ is the kinematic viscosity, $\nabla_S\, u=\frac12(\nabla u + (\nabla u)^\trans) \in \dabR^3 \times \dabR^3$ is the symmetric gradient, $\bff:\overline{\Omega}\rightarrow\dabR^3$ are volume forces acting on the fluid and $\bsigma:\partial\Omega\rightarrow\dabR^3$ is a stress defined as:
\begin{equation}
\bsigma=(2\nu\, \nabla_S\,\bu -pI)\bn\quad \mbox{on } \partial\Omega.
\label{eq1:8}
\end{equation}
The second equation in \eqref{eq1:1} guarantees the incompressibility condition.\linebreak{} The~mappings $\bu_D:\overline{\gamma}_D\rightarrow \dabR^3$ and $\bsigma_N:\overline{\gamma}_N\rightarrow\dabR^3$ specify the given Dirichlet and Neumann boundary data. The adhesive function is denoted by $\kappa:\gamma_S\rightarrow\dabR^{+}$, and the threshold slip bound function is represented by $g:\gamma_S\rightarrow\dabR^{+}$. The~unit outward normal vector, denoted by $\bn=\bn(\bx)\in\dabR^3$, along with two unit tangential vectors $\bt_1=\bt_1(\bx)\in\dabR^3$ and $\bt_2=\bt_2(\bx)\in\dabR^3$, are chosen such that the triplet $\{\bn,\bt_1,\bt_2\}$ forms an orthonormal basis at $\bx\in\partial\Omega$. 

The values of the normal and tangential velocity $\bu$ along $\partial\Omega$ are given by $u_n=\bu\cdot\bn$ and $\bu_t=\left(\bu\cdot\bt_1, \bu\cdot\bt_2\right)$, respectively. Analogously, $\bsigma_n$, $\bsigma_t$ are defined.

The classical solution of the problem \eqref{eq1:1} is represented by the functions $\bu\in\left(C^2(\Omega)\right)^3$ and $p\in C^1(\Omega)$, which possess a continuous extension to the~boundary $\partial\Omega$, such that all relations in \eqref{eq1:1} are satisfied after the~substitution of $\bu$ and $p$.

The weak formulation of the problem \eqref{eq1:1} requires the following forms:
\begin{alignat}{2}
&\label{form:1a}a_0:\left(H^1(\Omega)\right)^3 \times \left(H^1(\Omega)\right)^3\rightarrow \dabR,\quad &&a_{0}(\bw,\bv)=2\nu\displaystyle\int_\Omega \nabla_S\,\bw:\nabla_S\,\bv\, \mathrm{d} x,\\ \vspace{3mm}
&\label{form:1b}a_{\kappa}:\left(H^1(\Omega)\right)^3 \times \left(H^1(\Omega)\right)^3\rightarrow \dabR,\quad &&a_{\kappa}(\bw,\bv)=a_0(\bw,\bv)+\kappa\int_{\gamma_S}\bw_t\cdot\bv_t\, \mathrm{d} s,\\ \vspace{3mm}
&b:L^2(\Omega)\times\left(H^1(\Omega)\right)^3\rightarrow \dabR, && b(q,\bw)=-\displaystyle\int_\Omega q(\nabla \cdot \bw)\, \mathrm{d}x,\\ \label{form:2} \vspace{3mm}
&l:\left(H^1(\Omega)\right)^3 \rightarrow \dabR, &&l(\bv)=\displaystyle\int_\Omega \bff\cdot \bv + \int_{\gamma_N}\bsigma_N\cdot\bv\, \mathrm{d} s,\\ \label{form:3} \vspace{3mm}
&j:\left(H^1(\Omega)\right)^3 \rightarrow\dabR,\quad &&j(\bv)=\displaystyle\int_{\gamma_S} g||\bv_t||\,\mathrm{d} s.
\end{alignat}
Furthermore, we define the velocity solution set:
$$V_{\bu_D}=\left\{\bv\in\left(H^1(\Omega)\right)^3:\bv=\bu_D\ \mbox{ on }\  \gamma_D,\ v_n=0\ \mbox{ on }\ \gamma_S\right\}.$$
We denote by $V_{0}$ the space $V_{u_D}$ with $u_D = 0$ on $\gamma_D$.
The weak formulation of problem \eqref{eq1:1} attains the following form:
\begin{equation}
\left.
\begin{aligned}
&\text{Find}(\bu,p)\in V_{\bu_D}\times L^2(\Omega)\ \text{such that for all}\ (\bv,q)\in V_{0}\times L^2(\Omega)\!\ \\
&a_{\kappa}(\bu,\bv-\bu)+b(p,\bv-\bu)+j(\bv)-j(\bu)\geq l(\bv-\bu),\\
&b(q,\bu)=0.
\end{aligned}
\right\}
\label{eq1:2}
\end{equation}
It is thus a variational inequality of the second kind coupled with an equality constraint that guarantees the fulfillment of the second equation in~\eqref{eq1:1}. 

The existence and uniqueness of a weak solution follow from the next theorem (\cite{haslinger-kucera-sassi-satek}). The weak formulation described in the paper is equivalent to \eqref{eq1:2}.

\begin{theorem} Let $f\in(L^2(\Omega))^3,\ \sigma_N\in\left(L^2(\gamma_N)\right)^3$, and $g,\kappa\in L^\infty(\gamma_S)$, $g\geq0$, $\kappa\geq 0$. Then the solution $(u,p)$ to \eqref{eq1:2} exists, and the velocity component $u$ is unique. If $\gamma_N\neq \emptyset$, then the pressure component $p$ is unique as well, and it is defined up to an additive constant.
\end{theorem}
\subsection{Mixed finite element method}
\label{discretization}

We approximate problem \eqref{eq1:2} using the mixed finite element method~\cite{Boffi-Brezzi-Fortin-2013} with the P1-bubble/P1 finite element pair (also called MINI) introduced by Arnold, Brezzi, and Fortin~\cite{arnold-brezzi-fortin}. This pair is known for satisfying the $\inf$-$\sup$ stability Babu\v{s}ka-Brezzi condition~\cite{Girault-Raviart-1986,Quarteroni-Valli-1994}, good approximation property, and a small degree of freedom. Matrices and vectors will be assembled using vectorized codes proposed by J. Koko in \cite{koko} and their improved versions from~\cite{kucera-arzt-koko}.

Let $\T^h$ be a regular partition of $\overline{\Omega}$ and consider its segment $T\in\T^h$. By $h>0$, we denote the largest diameter of $T\in\T^h$. $N(\Omega) =\{\bx_1, \dots, \bx_{n_p}\}$ is the~set of all nodes of $\T^h$ belonging to $\overline{\Omega}$, with the cardinality (size) $n_p=|N(\Omega)|$. Let us introduce integers $n_d=|N(\gamma_D)|$, $n_s=|N(\gamma_S)|$, where $N(\gamma_D)=\{\bx_i \in N(\Omega): \bx_i \in \overline{\gamma}_D\}$, and $N(\gamma_S) = \{\bx_i \in N(\Omega): \bx_i \in \overline{\gamma}_S \setminus \overline{\gamma}_D\}$, respectively. Finally, $n_u=|N(\overline{\Omega} \setminus \overline{\gamma}_D)|$ , and $n_t$ is the number of all $T \in \T^h$. 

The piecewise linear Courant basis functions $\phi_i: \overline{\Omega} \rightarrow \dabR$, are associated with each vertex of $T\in\T^h$ such that $\phi_i(\bx_j) = \delta_{ij}$, $1 \leq i,j \leq n_p$. These functions generate the space
$$W^h = \left\{v^h \in C(\overline{\Omega}): v^h_{|_T} \in P^1(T) \quad \forall T \in \T^h \right\},\quad \dim W^h = n_p.$$

The bubble functions $\phi_b^T: \overline{\Omega} \rightarrow \dabR$, which are non-zero only in one \(T \in \T_h\), are defined as follows:
$$\phi_b^T = 4^4 \prod_{j=1}^{4} \phi_{i(j)},$$
where $\{\bx_{i(1)}, \bx_{i(2)}, \bx_{i(3)}, \bx_{i(4)}\}$ are the vertices of $T$. These functions generate the space
$$B^h = \left\{ v^h \in C(\overline{\Omega}):\ v^h_{|T} = c_T \phi_b^T,\ c_T \in \dabR \quad \forall T \in \T^h \right\},\quad \dim B^h = n_t.$$
With the above-introduced two spaces, we associate the following spaces:
$$
\begin{aligned}
V^h  &= \left(W^h\oplus B^h\right)^3,\\
V^h_{\bu_D}  &= \left\{\bv^h \in V^h: \bv^h(\bx_i) = \bu_D(\bx_i) \quad \forall \bx_i \in N(\gamma_D),\ v^h_n(\bx_i) = 0 \quad \forall \bx_i \in N(\gamma_S)\right\},\\
V^h_{0} &= V^h_{\bu_D} \ \text{with}\ \bu_D=0,\quad \dim V^h_{\bu_D} = \dim V^h_{0} = 3(n_u + n_t).
\end{aligned}
$$
The mixed finite element approximation of the problem \eqref{eq1:2} reads now as follows:
\begin{equation}
\left.
\begin{aligned}
&\text{Find }(\bu^h,p^h)\in V^h_{\bu_D}\times W^h\ \text{such that for all}\ (\bv^h,q^h)\in V^h_{0}\times W^h\!\ \\
&a_{\kappa}(\bu^h,\bv^h-\bu^h)+b(p^h,\bv^h-\bu^h)+j(\bv^h)-j(\bu^h)\geq l(\bv^h-\bu^h),\\
&b(q^h,\bu^h)=0.
\end{aligned}
\right\}
\label{eq1:3}
\end{equation}
\subsection{Algebraic formulation} The vectorized assembly functions of the finite element matrices and vectors representing the forms $a_{\kappa}$, $b$, and $l$ are described in \cite{kucera-arzt-koko,KAK-2024}. Let us note that during matrix construction, bubble components at the element level are effectively eliminated. 

We introduce first the algebraic vectors
$$
\bu = \begin{pmatrix} \bu^h(\bx_1) \\ \vdots \\ \bu^h(\bx_{n_u}) \end{pmatrix}\in\dabR^{3n_u} \text{, and} \quad \bp = \begin{pmatrix} p^h(\bx_1) \\ \vdots \\ p^h(\bx_{n_p}) \end{pmatrix}\in\dabR^{n_p}
$$
of the nodal velocities and pressure components, and the index set~\linebreak{} $\N:=\{1,2,\ldots,n_s\}$, which is the index set of the nodes lying on $\overline{\gamma}_S\setminus \overline{\gamma}_D$. {The index set of the remaining nodes is denoted by $\I:=\{n_s+1,\ldots,n_u\}$.}

Problem \eqref{eq1:3} can be written as searching a solution $(\bu,\bp)\in\dabR^{3n_u}\times\dabR^{n_p}$ of the generalized equation (GE)
\begin{equation}\label{EqGE0}
0\in F(\bu,\bp) = \left(\begin{array}{cc}A_{\kappa}& B^\trans \\ B & -E\end{array}\right) \left(\begin{array}{c} \bu\\  \bp \end{array}\right) - \left(\begin{array}{c} b \\ c\end{array}\right) + \partial q(\bu,\bp), 
\end{equation}

with $$q(\bu,\bp)=\sum_{i\in\N} \left(g_i||T^i\bu^i||+\delta_{S^i}(\bu^i)\right),\quad S^i:=\left\{\bv\in\dabR^3 :N^i\bv=0\right\},$$
where the indicator function $\delta_S$ is given by
$$\delta_S(\bx)=\left\{\begin{aligned}
0 \quad \text{if } \bx\in S,\cr
\infty\quad \text{if } \bx\notin S,
\end{aligned}\right.$$ and ensures thus the impermeability condition.

The matrix $A_{\kappa}\in\dabR^{3n_u\times3n_u}$ is nonsingular, symmetric and positive definite and $B\in\dabR^{n_p\times 3n_u}$ has the full row-rank. Further, the matrix $E\in\dabR^{n_p\times n_p}$ and the vector $c\in\dabR^{n_p}$ are products of the bubble compoments elimination. Moreover, $E$ is symmetric {and positive semidefinite}.

The matrices $T^i$ and $N^i$ are $2\times 3$ and $1\times 3$ matrices, respectively, containing the two tangents and the normal to the boundary at the i-th node. Further, $g_i=\mu(\bx_i)g(\bx_i)$, where $\mu(\bx_i)$ is the measure of the element. Note that the $3\times 3 $ matrix
$$Q^i:=\left(\begin{array}{c} T^i\\ N^i\end{array}\right)$$
must be orthogonal. Note that due to~\cite{haslinger-makinen-stebel} the generalized equation GE~\eqref{EqGE0} has a unique solution.

\section{SCD semismooth Newton* method}
After substituting $\bp=-\hat \bp$, problem \eqref{EqGE0} can be written as finding a~solution $(\bu,\hat\bp)$ of the GE
\begin{equation}\label{EqGE}
0\in F(\bu,\hat\bp) = \left(\begin{array}{cc}A_{\kappa}&-B^\trans \\ B & E\end{array}\right) \left(\begin{array}{c} \bu\\  \hat\bp \end{array}\right) - \left(\begin{array}{c} b \\ c\end{array}\right) + \partial q(\bu,\hat\bp).
\end{equation}
By setting $x:=\left(\begin{smallmatrix}u\\\hat p\end{smallmatrix}\right)\in\R^m$, $m:= n_u+n_p$ and
\[H(x):=\begin{pmatrix}A_{\kappa}&-B^\trans \\ B & E\end{pmatrix}\begin{pmatrix} u\\  \hat p \end{pmatrix} - \begin{pmatrix} b \\ c\end{pmatrix},\]
the GE \eqref{EqGE} attains the form
\begin{equation}\label{EqGE1}0\in F(x) = H(x)+\partial q(x)\end{equation}
with an affine linear operator $H:\R^m\to\R^m$ and a proper convex lower semicontinuous (lsc) function $q:\R^m\to\oR:=\R\cup\{\infty\}$. Since for all $x_i=(u_i,\hat p_i)$, $i=1,2$, we have
\[\skalp{H(x_1)-H(x_2),x_1-x_2}=\skalp{A_\kappa(u_1-u_2),u_1-u_2}+\skalp{E(\hat p_1-\hat p_2),\hat p_1-\hat p_2}\geq 0,\]
mapping $H$ is maximally monotone.

In order to solve  problems of this type one may apply a globally convergent algorithm based on the SCD \ssstar Newton method as presented in the recent paper by Gfrerer, Outrata and Valdman \cite{gfrerer-outrata-valdman}. Let us briefly describe this method, which is based on the SC (subspace containing) derivative introduced in \cite{GfrOut22}. In what follows, we will extensively work with $m$-dimensional subspaces of $\R^m\times\R^m$. Note that for two $m\times m$ matrices $A,B$ the subspace
\[\rge(A,B):=\{(Ap,Bp)\mv p\in\R^m\}\subset \R^m\times\R^m\]
has dimension $m$ if and only if the $2m\times m$ matrix $\left(\begin{smallmatrix}A\\B\end{smallmatrix}\right)$ has full column rank $m$.

If a single-valued mapping $G:\R^m\to\R^m$ is Fr\'echet differentiable at a~point $\xb$, the tangent cone to the graph of $G$ at $(\xb,G(\xb))$ is the $m$-dimensional subspace $\rge(I,\nabla G(\xb))$. This observation is the motivation for the following definition.
\begin{definition}
The set-valued mapping $G:\R^m\tto\R^m$ is called \emph{graphically smooth of dimension $m$} at $(x,y)\in\gph G$, if the tangent cone $T_{\gph G}(x,y)$ is an $m$-dimensional subspace of $\R^m\times \R^m$. Furthermore, $\OO_G$ denotes the set of all points $(x,y)\in\gph G$ such that $G$ is graphically smooth of dimension $m$ at $(x,y)$.
\end{definition}

Next, we consider the metric space $\Z_m$ of all $m$-dimensional subspaces of $\R^m\times\R^m$, equipped with the metric
    \begin{equation*}
        d_{\Z}(L_1,L_2):=\norm{P_1-P_2} \,,
    \end{equation*}
where $P_i$, $i=1,2$ denote the orthogonal projections onto the subspaces $L_i$. 

Given a subspace $L\in \Z_{m}$, we define its adjoint subspace $L^*\in\Z_{m}$ by
\[L^*=\{(v^*,u^*)\mv (u^*,-v^*)\in L^\perp\}.\]
Note that $(L^*)^*=L$ and $d_\Z(L_1,L_2)=d_\Z(L_1^*,L_2^*)$.

Next consider the definition of \emph{subspace containing (SC) derivatives}.
\begin{definition}\label{DefSCDProperty} Let $G:\R^m\tto\R^m$ be a set-valued mapping. Then:
  \begin{enumerate}
    \item The \emph{subspace containing (SC) derivative} $\Sp G:\gph G\tto \Z_m$ of $G$ is defined by
    \begin{align*}
        \Sp G(x,y):=\{L\in \Z_m\mv \exists \, (x_k,y_k)\longsetto{{\OO_G}}(x,y):\,&\\ \lim_{k\to\infty} d_\Z\big(L,T_{\gph G}(x_k,y_k)\big)&=0 \}.
    \end{align*}
    \item The \emph{adjoint SC derivative} $\Sp^* G:\gph G\tto \Z_m$ of $G$ is given by
    \[\Sp^*G(x,y) = \{L^*\mv L\in\Sp G(x,y)\}.\]
    \item The mapping $G$ is said to have the \emph{subspace containing derivative (SCD) property at} $(x,y)\in\gph G$, if there holds $\Sp G(x,y)\neq \emptyset$. Furthermore, we say that $G$ has the \SCD property \emph{around} $(x,y)\in\gph G$, if there is a neighborhood $W$ of $(x,y)$ such that $G$ has the \SCD property at every $(x',y')\in\gph G\cap W$. Finally, we call $G$ an \emph{\SCD mapping}, if $G$ has the \SCD property at every point of its graph.
  \end{enumerate}
\end{definition}

The SC derivative may be considered as an extension of the so-called {\em B-Jacobian (B-differential)} from single-valued to set-valued mappings. Given a mapping $H:D\to\R^m$, $D\subset\R^m$ open, let $\OO_H$ denote the set of all points belonging to $D$ where $H$ is \F differentiable. Then the B-Jacobian at a~point $x\in D$ is given by
\[\onabla H(x):=\{A\in \R^{m\times m}\mv \exists x_k\longsetto{\OO_H}x:\ A=\lim_{k\to\infty}\nabla H(x_k)\}.\]
It follows from \cite[Lemma 3.11]{GfrOut22} that
    \begin{equation*}
        \Sp H (x,H(x))\supset \{\rge(I,A)\mv A\in \onabla H(x)\} \,, \quad  \forall \, x\in D \,,
    \end{equation*}
and equality holds for all $x$ such that $H$ is Lipschitz continuous near $x$.

The set-valued mappings we are dealing with in this paper, are subdifferentials of convex lsc functions and for such mappings the SC derivative has a special structure. Further, the subspaces contained in the SC derivative offer a certain particularly useful basis representation.
\begin{definition}\label{DefZnP}
  We denote by $\ZmP$ the collection of all subspaces $L\in\Z_m$ such that there are symmetric $m\times m$ matrices $P$ and $W$ with the following properties:
  \begin{enumerate}
  \item[(i)] $L=\rge(P,W)$,
  \item[(ii)] $P^2=P$, i.e., $P$ represents the orthogonal projection onto some subspace of $\R^m$,
  \item[iii)] $W(I-P)=I-P$.
  \end{enumerate}
\end{definition}
Note that for symmetric matrices $W,P$ fulfilling conditions (ii) and (iii) of the above definition  we have
\begin{gather}\label{EqPW1}(I-P)W=\big(W(I-P)\big)^T=(I-P)^T=I-P=W(I-P),\\
\label{EqPW2}PW=WP=PWP,\\
\label{EqPW3}W=WP+W(I-P)=PWP+(I-P).
\end{gather}
By \cite[Lemma 3.3]{Gfr25a}, subspaces $L\in\ZmP$ are self adjoint, i.e., $L=L^*$. Since lsc convex functions are both prox-regular and subdifferentially continuous (\cite[Chapter 8]{RoWe98}), we may combine \cite[Corollary 3.2]{GfrOut22} and \cite[Proposition 3.6]{Gfr25a} to obtain the following statement.
\begin{proposition}
  For every lsc proper convex function $q:R^m\to\oR$ the subdifferential mapping $\partial q$ is an \SCD mapping and for every $(x,x^*)\in\gph\partial q$ there holds $\Sp(\partial q)(x,x^*)= \Sp^*(\partial q)(x,x^*)\subset\ZmP$.
\end{proposition}

  In what follows, we denote for a convex lsc function $q:\R^m\to\oR$ and a pair $(x,x^*)\in\gph \partial q$ by $\M_{P,W}\partial q(x,x^*)$ the collection of all symmetric $m\times m$ matrices $P$ and $W$ fulfilling conditions (ii) and (iii) of Definition \ref{DefZnP} and such that $\rge(P,W)\in \Sp(\partial q)(x,x^*)$.  By convexity of $q$,  $W$ is positive semidefinite for every $(P,W)\in\M_{P,W}\partial q(x,x^*)$. Further note that for twice  differentiable functions $q$ we have $\M_{P,W}\partial q(x,\nabla q(x))=\{(I,\nabla^2 q(x))\}$.

  Finally we need the definition of SCD semismoothness.
  \begin{definition}[{\cite[Definition 5.1]{GfrOut22}}]\label{DefSCDssstar}
We say that $G:\R^m\tto\R^m$ is {\em\SCD \ssstar at} $(\xb,\yb)\in\gph G$ if $G$ has the \SCD property around $(\xb,\yb)$ and
for every $\epsilon>0$ there is some $\delta>0$ such that
\begin{align}\label{EqDefSSCSemiSmooth}
\vert \skalp{x^*,x-\xb}-\skalp{y^*,y-\yb}\vert&\leq \epsilon
\norm{(x,y)-(\xb,\yb)}\norm{(x^*,y^*)}
\end{align}
holds for all $(x,y)\in \gph G\cap \B_\delta(\xb,\yb)$ and all $(y^*,x^*)\in L^*$, $L^*\in\Sp^*G(x,y)$.
\end{definition}

It can be easily shown that $G$ is \SCD \ssstar at $(\xb,\yb)$ if and only if
\begin{equation}\label{EqSCDssstar1}\lim_{(x,y)\longsetto{\gph G}(\xb,\yb)}\frac{\sup\{\dist{(\xb-x,\yb-y),L}\mv L\in \Sp G(x,y)\}}{\norm{(\xb-x,\yb-y)}}=0,\end{equation}
cf. \cite[Proposition 2.12]{Gfr25a}.

Many mappings appearing in applications are SCD \ssstar\hspace{-2mm}. For instance, if $\gph G$ is a closed subanalytic set then it follows from a result of Jourani \cite{Jou07}, that $G$ is SCD \ssstar at every point belonging to its graph.

  Let us now briefly describe  the SCD \ssstar Newton method for solving problem \eqref{EqGE1}, where we formulate the method for continuously differentiable mappings $H$. Let $\xb$ denote a solution to \eqref{EqGE1} and assume that $\partial q$~is~SCD \ssstar at $(\xb,-H(\xb))\in \gph\partial q$. Consider an iterate $\ee xk$ close to $\xb$. In a first step, the so-called {\bf approximation step}, we compute a pair $(\ee zk,\ee{z^*}k)\in\gph \partial q$ by
  \[\ee zk :={\rm prox}_{\lambda q}(\ee xk-\lambda H(\ee xk)),\quad \ee {z^*}k:=\frac{\ee xk-\ee zk}\lambda-H(\ee xk),\]
  where $\lambda>0$ is a fixed parameter and
  \[{\rm prox}_\varphi(x):=\arg \min_z\left\{\frac 12\norm{z-x}^2+\varphi(z)\right\}\]
  denotes the {\em proximal mapping} of an lsc function $\varphi:\R^m\to\oR$. Since $q$ is convex lsc, $\partial (\lambda q)$ is maximally monotone and thus both ${\rm prox}_{\lambda q}$ and\linebreak{} $Q:=Id-{\rm prox}_{\lambda q}$ are  single-valued and nonexpansive by Minty's theorem, cf.
  \cite[Theorem 12.15]{RoWe98}. Note that the computation of $\ee zk$ amounts exactly to one step of the well-known {\em forward-backward splitting method}, see, e.g., \cite{BauCom17}. Together with the identities $\xb={\rm prox}_{\lambda q}(\xb-\lambda H(\xb))$,\linebreak{} $H(\xb)=-Q(\xb-\lambda H(\xb))/\lambda$ and $\ee{z^*}k=Q(\ee xk-\lambda H(\ee xk))/\lambda$ and
   we obtain the bounds
  \[\norm{\ee zk-\xb}\leq\norm{\ee xk-\lambda H(\ee xk)-(\xb-\lambda H(\xb))}\leq (1+\lambda l)\norm{\ee xk-\xb}\]
  and
  \begin{align*}\norm{\ee {z^*}k+H(\xb)}&=\norm{Q(\ee xk-\lambda H(\ee xk))-Q(\xb-\lambda H(\xb))}/\lambda\\
  &\leq\norm{\ee xk-\lambda H(\ee xk)-(\xb-\lambda H(\xb))}/\lambda  \leq (\frac 1\lambda+ l)\norm{\ee xk-\xb},\end{align*}
  where $l$ denotes the Lipschitz constant of $H$ in some ball around $\xb$ containing the point $\ee xk$.
  Next we choose an arbitrary pair 
  \[(\ee Pk,\ee Wk)\in \M_{P,W}\partial q(\ee zk,\ee {z^*}k).\] 
  Then, by taking into account \eqref{EqSCDssstar1}, there is some $\ee pk\in\R^m$ such that
  \begin{align}\label{EqErrSemism}\lefteqn{\norm{(\xb-\ee zk-\ee Pk\ee pk, -H(\xb)-\ee{z^*}k-\ee Wk\ee pk)}}\\
   \nonumber&=\dist{(\xb-\ee zk,-H(\xb)-\ee {z^*}k) ,\rge(\ee Pk,\ee Wk)}\\\nonumber&=\oo(\norm{(\ee zk-\xb,\ee{z^*}k+H(\xb))})
   =\oo(\norm{\ee xk-\xb}).\end{align}
  By differentiability of $H$  we have
  \[H(\ee xk)-H(\xb)+\nabla H(\ee xk)(\xb-\ee xk)=\oo(\norm{\ee xk-\xb})\]
  and, since we can derive from \eqref{EqErrSemism} the relationships $\xb -\ee xk =\ee zk - \ee xk +\ee Pk\ee pk+\oo(\norm{\ee xk-\xb})$ and
  $-H(\xb)=\ee{z^*}k+\ee Wk\ee pk+\oo(\norm{\ee xk-\xb})$, the~equation
  \[H(\ee xk)+\ee{z^*}k+\ee Wk\ee pk+ \nabla H(\ee xk)(\ee zk-\ee xk +\ee Pk\ee pk)=\oo(\norm{\ee xk-\xb})\]
  follows. By multiplying with $\ee Pk$ and by using \eqref{EqPW2} and the fact that $\ee Pk$ represents an orthogonal projection, we obtain that
  \begin{align*}\lefteqn{\big(\ee Pk\nabla H(\ee xk)\ee Pk+\ee Wk\big)\ee Pk\ee pk}\\
   &= -\ee Pk\big(H(\ee xk)+\nabla H(\ee xk)(\ee zk -\ee xk)+\ee{z^*}k\big)+\oo(\norm{\ee xk-\xb})\\
   &= -\ee Pk\big(\nabla H(\ee xk)-I/\lambda\big)(\ee zk -\ee xk)+\oo(\norm{\ee xk-\xb}).
   \end{align*}
   Thus, if the matrix $\ee Pk\nabla H(\ee xk)\ee Pk+\ee Wk$ is invertible and the norm of its inverse is uniformly bounded for $\ee xk$ close to $\xb$, it follows that the solution $\ee {\triangle z}k$ of the system
   \begin{equation}\label{EqNewtonDirZ}\big(\ee Pk\nabla H(\ee xk)\ee Pk+\ee Wk\big)\triangle z=-\ee Pk\big(\nabla H(\ee xk)-I/\lambda\big)(\ee zk -\ee xk)\end{equation}
   satisfies
   \[\triangle \ee zk = \ee Pk\ee pk+\oo(\norm{\ee xk-\xb}) =\xb-\ee zk+ \oo(\norm{\ee xk-\xb}).\]
   By setting $\ee x{k+1}:= \ee zk+\triangle\ee zk$, we obtain finally that
   \[\norm{\ee x{k+1}-\xb}=\oo(\norm{\ee xk-\xb}).\]
   \begin{remark}
     By using  the properties of $\ee Pk$ and $\ee Wk$, it is easy to see that the difference $\ee {\triangle x}k:=\ee x{k+1}-\ee xk = \ee zk -\ee xk +\ee {\triangle z}k$ is a solution of the~linear system
     \begin{equation}\label{EqNewtonDirX}\big(\ee Pk\nabla H(\ee xk)+\ee Wk\big)\triangle x =(\ee Wk +\ee Pk/\lambda)(\ee zk-\ee xk).\end{equation}
     This formula corresponds to the computation of the Newton direction in \cite{gfrerer-outrata-valdman}, where also other basis representations of subspaces $\ee Lk\in\Sp(\partial q)(\ee zk,\ee{z^*}k)$ are considered.
   \end{remark}
   The following theorem is implied by \cite[Theorem 5.2]{gfrerer-outrata-valdman}:
   \begin{theorem}\label{ThConv}
     Assume that $\xb\in F^{-1}(0)$ is a solution to the GE \eqref{EqGE1} and assume that $\partial q$ is SCD \ssstar at $(\xb,-H(\xb))$. Further suppose that for every pair $(P,W)\in\M_{P,W}\partial q(\xb,-H(\xb))$ the matrix $P\nabla H(\xb)P+W$ is nonsingular. Then for every $\lambda>0$ there is a neighborhood $U$ of $\xb$ such that for every starting point $\ee x0\in U$ the procedure above is well-defined and generates a sequence $\ee xk$ converging superlinearly to $\xb$.
   \end{theorem}
   
   In order to globalize the SCD \ssstar Newton method, we use a~hybrid technique with a globally convergent method for finding zeros of the~sum of two maximally monotone operators, e.g., the well-known Douglas-Rachford method, cf. \cite{LiMe79}. In our case, one iterate of the Douglas-Rachford method  can be written down in the form $\ee x{k+1}=\T_\lambda^{DR}( \ee xk)$ with
   \[\T_\lambda^{DR}(x)= (I+\lambda H)^{-1}\big({\rm prox}_{\lambda q}(x-\lambda H(x))+\lambda H(x)).\]
   In this framework we use the residual
   \[r_\lambda(x):=(1+\frac 1\lambda)\norm{x-{\rm prox}_{\lambda q}(x-\lambda H(x))}.\]
   Note that $x$ is a solution for \eqref{EqGE1} if and only if $r_\lambda(x)=0$.
   Given the iterate $\ee xk$, we perform a line-search along the direction $\ee{\triangle x}k$ given either by the~solution of \eqref{EqNewtonDirX} or by the formula $\ee{\triangle x}k = \ee zk-\ee xk +\ee{\triangle z}k$. We are looking for the first index $j=0,1,\dots,N_\alpha$
   such that
   \[r_\lambda(\ee xk +2^{-j}\ee{\triangle x}k)\leq (1-\omega 2^{-j})r_\lambda(\ee xk)\]
   and set $\ee x{k+1}:=\ee xk+ 2^{-j}\ee{\triangle x}k$,
   where $\omega\in(0,1)$ and $N_\alpha\in {\mathbb{N}}$ are fixed parameters. If this fails, we compute the next iterate as
   \[\ee x{k+1}=\T_\lambda^{DR}(\ee  xk).\]
   We refer to \cite[Algorithm 3]{gfrerer-outrata-valdman} for a more advanced implementation of this procedure and its convergence properties.

\section{On the SCD \ssstar Newton method for the Stokes problem with Navier-Tresca boundary condition}

In this section we want to provide the implementation details for GE \eqref{EqGE} by the SSSN method. We assume that our unknowns $x=\left(\begin{smallmatrix}u\\\hat p\end{smallmatrix}\right)\in\R^{3n_u+n_p}$ are arranged as
\begin{align*}x=\begin{pmatrix}
  u_\N\\u_\I\\\hat p
\end{pmatrix}\mbox{ with } u_\N = \begin{pmatrix}u^1\\\vdots\\u^{n_s}\end{pmatrix},\ u_\I = \begin{pmatrix}u^{n_s+1}\\\vdots\\u^{n_u}\end{pmatrix},\\[3mm] u^i\in\R^3,\ i=1,\ldots,n_u,\ \hat p\in\R^{n_p}. \end{align*}

In our case the convex function $q$ has the separable structure
\begin{align*}q(u,\hat p)=\sum_{i\in\N}q_i(u^i)\quad\mbox{with}\quad q_i(u^i)=g_i||T^i u^i||+\delta_{S^i}(u^i),\\ S^i:=\left\{v\in\R^3: N^i v=0\right\}\end{align*}
and thus we have
\begin{equation}\label{EqSubdiff_q}\partial q(u,\hat p)=\prod_{i\in\N}\partial q_i(u^i)\times\{0_{3(n_u-n_s)}\}\times \{0_{n_p}\}\end{equation}
and
\[{\rm prox}_{\lambda q}(u,\hat p)=\begin{pmatrix}\prod_{i\in \N}{\rm prox}_{\lambda q_i}(u^i)\\u_\I\\\hat p\end{pmatrix}.\]
Further, it follows from \cite[Lemma 3.8]{gfrerer-outrata-valdman} that for any $(u^*,\hat p^*)\in\partial q(u,\hat p)$ and any $(P,W)\in\Sp(\partial q)\big((u,\hat p),(u^*,\hat p^*)\big)$, the matrices $P$ and $W$ are block diagonal matrices of the form
\begin{equation}\label{EqPW_representation} P=\begin{pmatrix}
  P_\N&&\\
  &P_\I&\\
  &&P_p
\end{pmatrix},\quad
W=\begin{pmatrix}
  W_\N&&\\
  &W_\I&\\
  &&W_p
\end{pmatrix}\end{equation}
with $P_\I = I_{3(n_u-n_s)}$, $W_\I=0_{3(n_u-n_s)\times 3(n_u-n_s)}$, $P_p=I_{n_p}$,  $W_p=0_{n_p\times n_p}$ and
\begin{equation}\label{EqPN_WN}P_\N=\begin{pmatrix}
  P^{1}&&\\
  &\ddots&\\
  &&P^{n_s}
\end{pmatrix},\quad W_\N=\begin{pmatrix}
  W^{1}&&\\
  &\ddots&\\
  &&W^{n_s}
\end{pmatrix},\end{equation}
where $(P^i,W^i)\in\M_{P,W}\partial q_i(u^i,{u^*}^i)$, $i\in \N$.
Conversely, for every collection of matrices $(P^i,W^i)\in \M_{P,W}\partial q_i(u^i, {u^*}^i)$, $i\in\N$, we can assemble a pair $(P,W)\in\M_{P,W}\partial q\big( (u,\hat p),(u^*,\hat p^*)\big)$ in this way.

Hence, in order to compute the quantities required for the SCD \ssstar Newton method, we only have to compute their counterparts for the $3$-dimensional functions $q_i$, $i\in\N$. By elementary calculus rules of convex analysis we immediately obtain
\begin{equation}\label{EqSubdiff_qi}\partial q_i(u)=\begin{cases}\emptyset&\mbox{if $N^i u\not=0$,}\\
g_i{T^i}^\trans\B_2+\rge N_i^\trans&\mbox{if $T^iu=0$ and $N^i u=0$, i.e., $u=0$}\\
  g_i{T^i}^\trans\frac{T^iu}{\norm{T^i u}}+\rge {N^i}^\trans&\mbox{if $N^i u=0$, $T^iu\not=0$,}
  \end{cases}\end{equation}
  where $\B_2$ denotes the unit ball in $\R^2$.
  
  In what follows we frequently use the following properties of $T^i$ and $N^i$, which follow immediately from the orthogonality of $Q^i$:   \[T^i{T^i}^\trans = I_2,\ N^i{N^i}^\trans = 1,\ T^i{N^i}^\trans = 0_2,\ {T^i}^\trans T^i + {N^i}^\trans N^i=I_3.\]
  Further, the symmetric $3\times 3$ matrices ${T^i}^\trans T^i$ and ${N^i}^\trans N^i$ represent the orthogonal projections onto $S^i = \ker N^i$ and ${S^i}^\perp= \rge {N^i}^\trans$ and for every $u\in S^i$ there holds $\norm{u}=\norm{T^i u}$.

  \begin{lemma}\label{LemProx}For every $\lambda>0$ and every $u\in\R^3$ there holds
    \[{\rm prox}_{\lambda q_i}(u) = \begin{cases}0&\mbox{if $\norm{T^i u}\leq \lambda g_i$,}\\
    \Big(1-\frac{\lambda g_i}{\norm{T^i u}}\Big){T^i}^\trans T^i u&\mbox{else.}
    \end{cases}\]
  \end{lemma}
  \begin{proof}
    By taking into account the relationship $S^i = \ker N^i=\rge {T^i}^\trans$, we obtain that ${\rm prox}_{\lambda q_i}(u)={T^i}^\trans z$ with
    \begin{align*}z\in\argmin\left\{\frac 1{2\lambda}\norm{{T^i}^\trans z-u}+g_i\norm{T^i{T^i}^\trans z}\right\}\\ =\argmin\left\{\frac 1{2\lambda}\norm{z-T^iu}+g_i\norm{z}\right\}\\
    ={\rm prox}_{\lambda g_i\norm{\cdot}}(T^i u),
    \end{align*}
    where we have used the equality $T^i{T^i}^\trans=I_2$. By using the well-known formulas for the proximal mapping of the Euclidean norm, see, e.g. \cite[Example 6.19]{Be17}, the assertion follows.
      \end{proof}

      In what follows we partition $\gph\partial q_i$ into 3 sets
      \[\gph \partial q_i=\Sigma^i_{\rm slip} \cup \Sigma^i_{\rm stick}\cup \Sigma^i_{\rm trans}\]
      with
      \begin{align*}
        \Sigma^i_{\rm slip}&=\{(u,u^*)\in\gph\partial q_i\mv u\not=0\},\\
        \Sigma^i_{\rm stick}&=\{(u,u^*)\in\gph\partial q_i\mv u=0,\ \norm{T^iu^*}<g_i\},\\
        \Sigma^i_{\rm trans}&=\{(u,u^*)\in\gph\partial q_i\mv u=0,\ \norm{T^iu^*}=g_i\}.
      \end{align*}
      Note that both $\Sigma^i_{\rm slip}$ and $\Sigma^i_{\rm stick}$ are relatively open in $\gph \partial q_i$, i.e., given $(u,u^*)\in \Sigma^i_{\rm slip}$ ($\Sigma^i_{\rm stick}$), then any $(v,v^*)\in\gph \partial q_i$ sufficiently close to $(u,u^*)$ also belongs to $\Sigma^i_{\rm slip}$ ($\Sigma^i_{\rm stick}$). This does not hold true for the ``transition part'' $\Sigma^i_{\rm trans}$ which is  the boundary of both the ``slipping part'' $\Sigma^i_{\rm slip}$ and the~``sticking part'' $\Sigma^i_{\rm stick}$.

      \begin{lemma}Let $(u, u^*)\in\gph\partial q_i$.
      \begin{enumerate}
        \item[(i)]If $(u,u^*)\in\Sigma^i_{\rm slip}$ then $\M_{P,W}\partial q_i(u, u^*)=\{(P,W)\}$ is a singleton with
        \[P={T^i}^\trans T^i,\ W=\frac {g_i}{\norm{T^i u}}{T^i}^\trans\Big(I_2-\frac{(T^i u)(T^i u)^\trans}{\norm{T^i u}^2}\Big)T^i+{N^i}^\trans N^i.\]
        \item[(ii)]If $(u,u^*)\in\Sigma^i_{\rm stick}$ then $\M_{P,W}\partial q_i(u, u^*)=\{(0_{3\times 3},I_3)\}$ is a singleton.
        \item[(iii)]If $(u,u^*)\in\Sigma^i_{\rm trans}$ then
        $\M_{P,W}\partial q_i(u, u^*)\supset\{(0_{3\times 3},I_3)\}$.
      \end{enumerate}
      Further, $\OO_{\partial q_i}\supset\Sigma^i_{\rm slip}\cup \Sigma^i_{\rm stick}$.
      \end{lemma}
      \begin{proof}
        ad(i): Since $u\not = 0$ and $N^iu=0$, there must hold $T^iu\not=0$. Using  \eqref{EqSubdiff_qi} we obtain that
        \begin{align*}&\lefteqn{T_{\gph \partial q_i}(u, u^*)}\\
        &=\left\{\left(v, \frac {g_i}{\norm{T^i u}}{T^i}^\trans\Big(I_2-\frac{(T^i u)(T^i u)^\trans}{\norm{T^i u}^2}\Big)T^i v + {N^i}^\trans w\right)\bmv N^iv=0, w\in\R\right\}\end{align*}
        and by taking into account that the orthogonal projections onto $\ker N^i$ and $\rge {N^i}^\trans$ are represented by ${T^i}^\trans T^i$ and
        ${N^i}^\trans N^i$, the formula
        \begin{align*}&\lefteqn{T_{\gph \partial q_i}(u, u^*)}\\
        &=\rge\Big({T^i}^\trans T^i,\frac {g_i}{\norm{T^i u}}{T^i}^\trans\Big(I_2-\frac{(T^i u)(T^i u)^\trans}{\norm{T^i u}^2}\Big)T^i+{N^i}^\trans N^i\Big)\in \Z_d\end{align*}
        follows. Hence, $(u,u^*)\in\OO_{\partial q_i}$ and therefore $\Sigma^i_{\rm slip}\subset\OO_{\partial g_i}$.  For every sequence $(u_k,u_k^*)\longsetto{\gph \partial q_i}(u,u^*)$ we have $(u_k,u_k^*)\in \Sigma^i_{\rm slip}\subset\OO_{\partial g_i}$ for all $k$ sufficiently large and, together with $\lim_{k\to\infty}T^iu_k=T^i u$, we conclude that $\lim_{k\to\infty} T_{\gph \partial q_i}(u_k, u_k^*)=T_{\gph \partial q_i}(u, u^*)$ in $\Z_3$. This verifies our assertion.

        ad(ii): Since $u^*={T^i}^\trans(T^iu^*) +{N^i}^\trans (N^iu^*)$ and $T^iu^*\in\inn(g_i\B_2)$, we conclude from \eqref{EqSubdiff_qi} that
        \[T_{\gph \partial q_i}(u, u^*)=\{(0,{T^i}^\trans v +{N^i}^\trans w)\mv v\in \R^2, w\in\R\}=\{0\}\times\R^3\in\Z_3.\]
        Now the assertion together with $\Sigma^i_{\rm stick}\subset\OO_{\partial q_i}$ easily follows.

        ad(iii): Since the sequence 
        \begin{align*}(u_k,u_k^*)&=(0,(1-\frac 1k){T^i}^\trans T^iu^*+{N^i}^\trans N^iu^*\longsetto{\Sigma_{\rm stick}}(u,u^*)
        \end{align*}
        and the tangent cone to $\gph\partial q_i$ at these points is always $\rge(0_{3\times 3},I_3)$, the~claim holds true.

      \end{proof}
      \begin{remark}In order to perform the SCD \ssstar Newton method it is sufficient to compute at each point belonging to $\gph \partial q_i$ {\bf just one} element of the SC derivative and we do not need to compute the whole SC derivative.
        For the sake of completeness we state (without proof) that actually\linebreak{} $\OO_{\partial q_i}=\Sigma^i_{\rm slip}\cup \Sigma^i_{\rm stick}$ and for every $(0,u^*)\in\Sigma^i_{\rm trans}$ there holds
        \[\M_{P,W}\partial q_i(0,u^*)=\{(0_{3\times 3},I_3), (vv^\trans,I_3-vv^\trans)\}\mbox{ with }v=\frac{{T^i}^TT^iu^*}{g_i}.\]
      \end{remark}
      We will now show that all the assumptions of Theorem \ref{ThConv} are fulfilled.

      Observe that $\gph \partial q_i$ has the representation
      \[\gph \partial q_i=\left\{(u,u^*)\in\R^3\times\R^3\bmv\exists v\in\R^2:\begin{array}{l}
      N^i u=0,\
      T^i u^*= g_i v, \norm{v}^2\leq 1,\\
      (v^\trans T^iu)^2=\norm{T^iu}^2,\ v^\trans T^iu\geq 0
      \end{array}\right\}\]
      and is therefore  a closed subanalytic set. Thus, $\partial q_i$ is SCD \ssstar at every point $(u,u^+)\in\gph\partial q_i$ and, by using \eqref{EqSubdiff_q} it is easy to see that the~same holds true for the whole mapping $\partial q$.

      Next, let $\bar x=(\bar u,\bar p)$ denote the solution to \eqref{EqGE} and we have to show that $P\nabla H(\xb)P+W$ is nonsingular for every pair $(P,W)\in\M_{P,W}\partial q(\bar x)$. To this aim we partition the matrices $A_\kappa$ and $B$ into
      \[A_\kappa =\begin{pmatrix}A_{\N\N}&A_{\I\N}^\trans\\A_{\I\N}&A_{\I\I}\end{pmatrix},\quad B=\begin{pmatrix}B_\N&B_\I\end{pmatrix},
      \]
      where $A_{\N\N}\in\R^{3n_s\times 3n_s}$, $A_{\I\N}\in\R^{3(n_u-n_s)\times 3n_s}$, $A_{\I\I}\in\R^{3(n_u-n_s)\times 3(n_u-n_s)}$,\linebreak{} $B_\N\in\R^{n_p\times 3n_s}$ and $B_\I\in\R^{n_p\times 3(n_u-n_s)}$. Now consider any $(P,W)\in\M_{P,W}\partial q(\bar x)$ and its representation \eqref{EqPW_representation}. It follows that
      \[P\nabla H(\xb)P+W =\begin{pmatrix}P_\N A_{\N\N}P_\N+W_\N&P_\N A_{\I\N}^\trans&-P_\N B_\N^\trans\\
      A_{\I\N}P_\N&A_{\I\I}&-B_\I^\trans\\
      B_\N P_\N& B_\I&E
      \end{pmatrix}.\]
      Consider now the blocks
      \[\hat A:= \begin{pmatrix}
        P_\N A_{\N\N}P_\N+W_\N&P_\N A_{\I\N}^\trans\\
      A_{\I\N}P_\N&A_{\I\I}
      \end{pmatrix},\quad \hat B:=\begin{pmatrix}
        B_\N P_\N&B_\I
      \end{pmatrix}.\]
      We first claim that $\hat A$ is positive definite. Assume on the contrary that there is a vector $z\not=0$ satisfying $z^\trans \hat A z\leq0$. By setting
      \[z:=\begin{pmatrix}
        z_\N\\z_\I
      \end{pmatrix},\  \hat z:= \begin{pmatrix}
        P_\N z_\N\\z_\I
      \end{pmatrix}\]
      we obtain that
      \[0\geq z^\trans\hat A z=\hat z^\trans A_\kappa \hat z+ z_\N^\trans W_\N z_\N.\]
      Thus, since $A_\kappa$ is positive definite and $W_\N$ is symmetric positive semidefinite, it follows that $z_\I=0$, $P_\N z_\N=0$ and $W_\N z_\N=0$. Hence $z_\N=(I-P_\N)z_\N$ and, by using \eqref{EqPW1}, we obtain that
      \[0= W_\N z_\N=W_\N(I-P_\N)z_\N=(I-P_\N)z_\N=z_\N\]
      and consequently $z=0$, a contradiction. This implies that $\hat A$ is positive definite and it follows that the matrix
      \[P\nabla H(\xb)P+W=\begin{pmatrix}\hat A&-\hat B^T\\\hat B&E\end{pmatrix}\] is nonsingular if and only if
      $\hat B\hat A^{-1}\hat B^\trans+E$ is nonsingular. Since $E$ is symmetric and positive semidefinite, this is in turn equivalent to the condition \begin{equation}\label{EqSurj}\hat B^\trans s=0,\ Es=0\ \Rightarrow\ s=0.\end{equation}
      Now consider the matrix
      \begin{equation}\label{EqMatrNoSlip}\begin{pmatrix}A_{\I\I}&-B_\I^\trans\\ B_\I&E\end{pmatrix}\end{equation}
       corresponding to the discretized Stokes problem with no-slip boundary conditions which has a unique solution. Hence, the matrix \eqref{EqMatrNoSlip} is nonsingular and our considerations above imply the condition $B_\I^\trans s=0,Es=0\Rightarrow s=0$. From this we can immediately deduce that \eqref{EqSurj} holds true and therefore the~matrix  $P\nabla H(\xb)P+W$ is nonsingular.

       We may conclude that  all the conditions of Theorem \ref{ThConv} are fulfilled and therefore the SCD \ssstar Newton method converges locally superlinearly to the solution of the Stokes problem with stick-slip boundary conditions.

       At the end of this section let us summarize all the steps necessary to implement the SCD \ssstar Newton method for solving \eqref{EqGE}. 
          
       \begin{enumerate}
       \item Choose $\lambda>0$, starting point 
       \[\ee x0=\left(\begin{pmatrix}u_\N\\u_I\\\hat p\end{pmatrix}\right)\in\R^{3n_u+n_p},\] tolerance $\epsilon>0$ and parameters $\omega\in(0,1)$, $N_\alpha\in \mathbb{N}$. Set the iteration counter $k=0$.
       \item Compute
       \[H(\ee x0)=:\begin{pmatrix}y^1\\\vdots\\y^{n_u}\\w\end{pmatrix}=:\begin{pmatrix}y_\N\\y_\I\\w\end{pmatrix}\]
       and
       \[z_\N:=\begin{pmatrix}{\rm prox}_{\lambda q_1}(u^1-\lambda y^1)\\\vdots\\ {\rm prox}_{\lambda q_{n_s}}(u^{n_s}-\lambda y^{n_s})\end{pmatrix},\ r_\N:=z_\N-u_\N\]
       by using Lemma \ref{LemProx} and set
       \[\ee r0:=r_\lambda(\ee x0)=\Big(1+\frac 1\lambda\Big)\sqrt{\lambda^2(\norm{y_\I}^2+\norm{w}^2)+\norm{r_\N}^2}.\]
       \item If $\ee rk\leq \epsilon\ee r0$ stop the algorithm.
       \item Compute for each $i\in\N$ the matrices
       \begin{align*}\hskip-0.6cm(P^i,W^i)&:=\begin{cases}
       \left(0_{3\times 3},I_3\right)&\hskip-0.8cm\mbox{if $z^i=0$,}\\
       \left({T^i}^\trans T^i,\frac {g_i}{\norm{T^i z^i}}{T^i}^\trans\Big(I_2-\frac{(T^i z^i)(T^i z^i)^\trans}{\norm{T^i z^i}^2}\Big)T^i+{N^i}^\trans N^i\right)&\mbox{else}\end{cases}
       \end{align*}
       and assemble the matrices $P_\N$, $W_\N$ according to \eqref{EqPN_WN}.
       \item Compute the Newton direction $\ee {\triangle x}k$ as solution of the system
       \begin{equation}\label{EqNewtSyst}
         \hskip-0.6cm\begin{pmatrix}P_\N A_{\N\N}+W_\N&P_\N A_{\I\N}^\trans&-P_\N B_\N^\trans\\
      A_{\I\N}&A_{\I\I}&-B_\I^\trans\\
      B_\N & B_\I&E
      \end{pmatrix}\triangle x=\begin{pmatrix}\big(W_\N+\frac1\lambda P_\N\big)r_\N\\-y_\I\\-w\end{pmatrix}\end{equation}
      and set $\alpha:=1$.
      \item Compute the trial point
      \[\tilde x:=\ee xk +\alpha\ee{\triangle x}k=:\begin{pmatrix}u_\N\\u_\I\\\hat p\end{pmatrix}\]
       together with
       \begin{align*}
       \hskip-2.2pt H(\tilde x)&=:\begin{pmatrix}y^1\\\vdots\\y^{n_u}\\w\end{pmatrix}=:\begin{pmatrix}y_\N\\y_\I\\w\end{pmatrix},\ 
       z_\N:=\begin{pmatrix}{\rm prox}_{\lambda q_1}(u^1-\lambda y^1)\\\vdots\\ {\rm prox}_{\lambda q_{n_s}}(u^{n_s}-\lambda y^{n_s})\end{pmatrix},\\ r_\N&:=z_\N-u_\N
       \end{align*}
       and set
       \[\tilde r:=r_\lambda(\tilde x)=\Big(1+\frac 1\lambda\Big)\sqrt{\lambda^2(\norm{y_\I}^2+\norm{w}^2)+\norm{r_\N}^2}.\]
       \item If $\tilde r\leq (1-\omega\alpha)\ee rk$ set 
       \[\ee x{k+1}:=\tilde x,\ \ee r{k+1}:=\tilde r.\]
       Increase the iteration counter $k:=k+1$ and go to step 3.
       \item If $\alpha>2^{-N_\alpha}$ set $\alpha:=\alpha/2$ and go to step 6.
       \item Set 
       \[\begin{pmatrix}u_\N\\u_\I\\ \hat p\end{pmatrix}:=\ee x{k+1}:=\T^{DR}_\lambda(\ee xk),\] i.e., we compute the next iterate by one step of the Douglas-Rachford method, and, after evaluating
           \begin{align*}&H(\ee x{k+1})=:\begin{pmatrix}y_\N\\y_\I\\w\end{pmatrix},\
       z_\N:=\begin{pmatrix}{\rm prox}_{\lambda q_1}(u^1-\lambda y^1)\\\vdots\\ {\rm prox}_{\lambda q_{n_s}}(u^{n_s}-\lambda y^{n_s})\end{pmatrix},\ r_\N:=z_\N-u_\N,\\
       &\ee r{k+1}:=r_\lambda(\ee x{k+1})=\Big(1+\frac 1\lambda\Big)\sqrt{\lambda^2(\norm{y_\I}^2+\norm{w}^2)+\norm{r_\N}^2},\end{align*}
       increase the iteration counter $k:=k+1$ and go to step 3.
      \end{enumerate}
      We omit the details of implementing the Douglas-Rachford method, because in our numerical experiments it was never encountered. 


\section{Numerical experiments}
The following section shows two examples where the SSSN method has been tested. In the first one, the convergence rate is examined on a simple unit cube region as a function of the mesh size and the settings of the stick-slip condition parameters. In the second example, the behavior of the method for various slip condition parameters is monitored on a more complex Cerebral Aneurysm region.

{
In our implementation we solved the linear system \eqref{EqNewtSyst} as follows. Denoting 
\[\triangle x=\begin{pmatrix}\triangle u_\N\\\triangle u_\I\\\triangle \hat p\end{pmatrix},\]
we eliminated 
\[\triangle u_\I = A_{\I\I}^{-1}\big(-y_\I-A_{\I\N}\triangle u_\N+B_\I^\trans\triangle \hat p\big)\]
to obtain the reduced system
\begin{align*}&\begin{pmatrix}P_\N(A_{\N\N}-A_{\I\N}^\trans A_{\I\I}^{-1}A_{\I\N})& -P_\N\big(B_\N^\trans- A_{\I\N}^\trans A_{\I\I}^{-1}B_\I^\trans\big)\\
B_\N-B_\I A_{\I\I}^{-1}A_{\I\N}& E+B_\I A_{\I\I}^{-1}B_\I^\trans\end{pmatrix}\begin{pmatrix}\triangle u_\N\\\triangle \hat p\end{pmatrix}\\
&\qquad\qquad\qquad\qquad\qquad\qquad=\begin{pmatrix}\big(W_\N+\frac1\lambda P_\N\big)r_\N+A_{\I\N}^\trans A_{\I\I}^{-1}y_\I\\
-w+B_I A_{\I\I}^{-1}y_\I\end{pmatrix}\end{align*}
which was solved by GMRES. In order to efficiently perform the matrix-times-vector multiplications needed by GMRES, we computed in advance a~Cholesky factorization of the sparse matrix $A_{\I\I}$.}

\subsection*{Example 1}
Let $\Omega=(0,1)^3$ with the boundary $\partial\Omega$ which consist of three non-empty disjoint parts such that $\partial \Omega = \overline{\gamma}_D \cup\overline{\gamma}_N\cup\overline{\gamma}_S$.  Zero Dirichlet boundary conditions are applied on $\gamma_D$, Neumann condition with $\bsigma_n=\nu\Delta\bu_{exp}\bv-p_{exp}\bv$ on $\gamma_N$, and impermeability and a stick-slip boundary condition are applied on $\gamma_S$. Function $\bff$ in \eqref{eq1:1} is given as $\bff=- 2\nu \, \nabla \cdot \nabla_S\,\bu_{exp} + \nabla p_{exp}$, 
where $\bu_{exp}=(u_{exp,1},u_{exp,2},u_{exp,3})$,
\begin{alignat*}{2}
&u_{ex}(x,y,z)&&= 4z(1-z)\sin(2\pi y)\left(1-\cos(2\pi x)\right),\\
&u_{ey}(x,y,z)&&= 4z(1-z)\sin(2\pi x)(\cos(2\pi y)-1),\\
&u_{ez}(x,y,z)&&=0,\\
&p_e(x,y,z)&&=2\pi\left(\cos(2\pi y)-\cos(2\pi x)-\cos(2\pi z)\right),
\end{alignat*}
and the direction $(x,y,z)=(x_1,x_2,x_3)$. Kinematic viscosity v=0.9. To demonstrate the significance of the choice of adhesive coefficient $\kappa$ and slip limit value $g$, experiments were performed for adhesive coefficient $\kappa = 5$ and three different values of slip limit: $g=0$ (for this choice of $\kappa$ and $g$, pure slip occurs in this example and the stick-slip boundary condition is described by the Navier boundary condition), $g=5$ (for this choice of $\kappa$ and $g$, simultaneous slip and adhesion occur in this example, and the stick-slip boundary condition is described by the Navier-Tresca boundary condition), and $g=10$ (for this choice of $\kappa$ and $g$, pure adhesion occurs in this example, and the solution to this example looks the same as if zero Dirichlet was applied on the boundary $\gamma_S$). Let us recall that different uniform meshes are characterized by parameters $n_p$ and $n_t$, which represent the number of nodes and tetrahedra, respectively. The number of nodes on $\gamma_S$ is described by $n_s$. The accuracy of Newton method is set to $\epsilon=1e^{-8}$ and the accuracy of the inner GMRES solver is controlled by the adaptive accuracy described for each Newton iteration $k>0$ as $$inner_{tol}^k:=\min(r_{tol}\times err^{k-1}, c_{fact}\times tol^{k-1}),$$ 
where $err$ is the error of the Newton method from the previous iteration with the constants $r_{tol}=0.95$ and $c_{fact}=0.8$. The scheme of the $\Omega$ and structure of the mesh are shown in the figures \ref{fig5:1}, \ref{fig5:2}, respectively.
\begin{figure}[ht!]
\begin{minipage}{.49\textwidth}
\centering
\begin{tikzpicture}[scale=3.0]

\fill[line width=2.pt,color=black,fill=ffffff,fill opacity=0.8] (0,0,1) -- (1,0,1) -- (1,0,0) -- (0,0,0) -- cycle;
\fill[line width=2.pt,color=black,fill=ffffff,fill opacity=0.8] (0,0,0) -- (1,0,0) -- (1,1,0) -- (0,1,0) -- cycle;
\fill[line width=2.pt,color=black,fill=blue,fill opacity=0.3] (0,0,0) -- (0,1,0) -- (0,1,1) -- (0,0,1) -- cycle;

\draw[dashed](0,1,0)--(0,0,0)--(1,0,0);
\draw[dashed](0,0,0)--(0,0,1);

\fill[line width=2.pt,color=black,fill=blue,fill opacity=0.3] (1,0,1) -- (1,0,0) -- (1,1,0) -- (1,1,1) -- cycle;
\fill[line width=2.pt,color=black,fill=red,fill opacity=0.3] (0,1,1) -- (1,1,1) -- (1,1,0) -- (0,1,0) -- cycle;
\fill[line width=2.pt,color=black,fill=ffffff,fill opacity=0.3] (0,0,1) -- (1,0,1) -- (1,1,1) -- (0,1,1) -- cycle;

\draw[ultra thick](0,0,1)--(0,1,1)--(1,1,1)--(1,0,1)--(0,0,1);
\draw[ultra thick](0,1,0)--(1,1,0)--(1,0,0);
\draw[ultra thick](0,1,0)--(0,1,1);
\draw[ultra thick](1,1,0)--(1,1,1.0);
\draw[ultra thick](1,0,0)--(1,0,0.98);

 \draw[font=\large] (0.35,0.6,0.5) node[below]{$\Omega$};
 \draw (0.4,1.0,0.2) node[below]{$\color{red}\gamma_S$};
 \draw (0.4,0,0.2) node[below]{$\gamma_D$};

 \draw (0.0,0.5,0.45) node[below]{$\color{blue}\gamma_N$};
  \draw (1,0.5,0.45) node[below]{$\color{blue}\gamma_N$};
   \draw (0.5,0.5,0.0) node[below]{$\gamma_D$};
    \draw (0.6,0.6,1) node[below]{$\gamma_D$};

\draw[->][thick] (xyz cs:x=-0.15) -- (xyz cs:x=1.2) node[below] {$x_1$};
\draw[->][thick] (xyz cs:y=-0.15) -- (xyz cs:y=1.2) node[left] {$x_2$};

\draw[->][thick] (0.07,0.07,-0.07) -- (-0.52,-0.52);
\draw (-0.41,-0.47) node [below]{$x_3$};

\end{tikzpicture}
 \caption{Scheme of $\Omega$}
 \label{fig5:1}
 \end{minipage}
\begin{minipage}{.49\textwidth}
\centering
\includegraphics[width=1.0\linewidth]{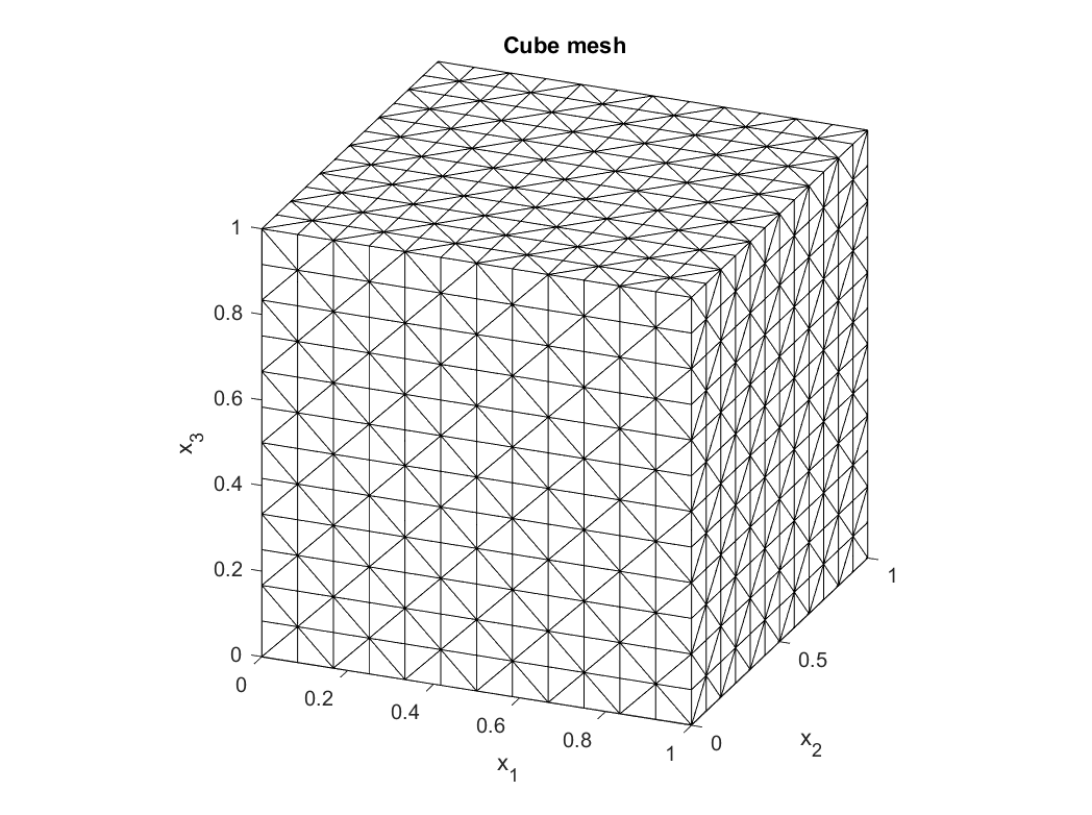}\\
\captionof{figure}{Cube mesh}
\label{fig5:2}
\end{minipage}
\end{figure}

Tangential velocity field is shown in the figures \ref{fig5:3} - \ref{fig5:4}.

\begin{figure}[ht!]
\begin{minipage}{.49\textwidth}
\centering
\includegraphics[width=\linewidth]{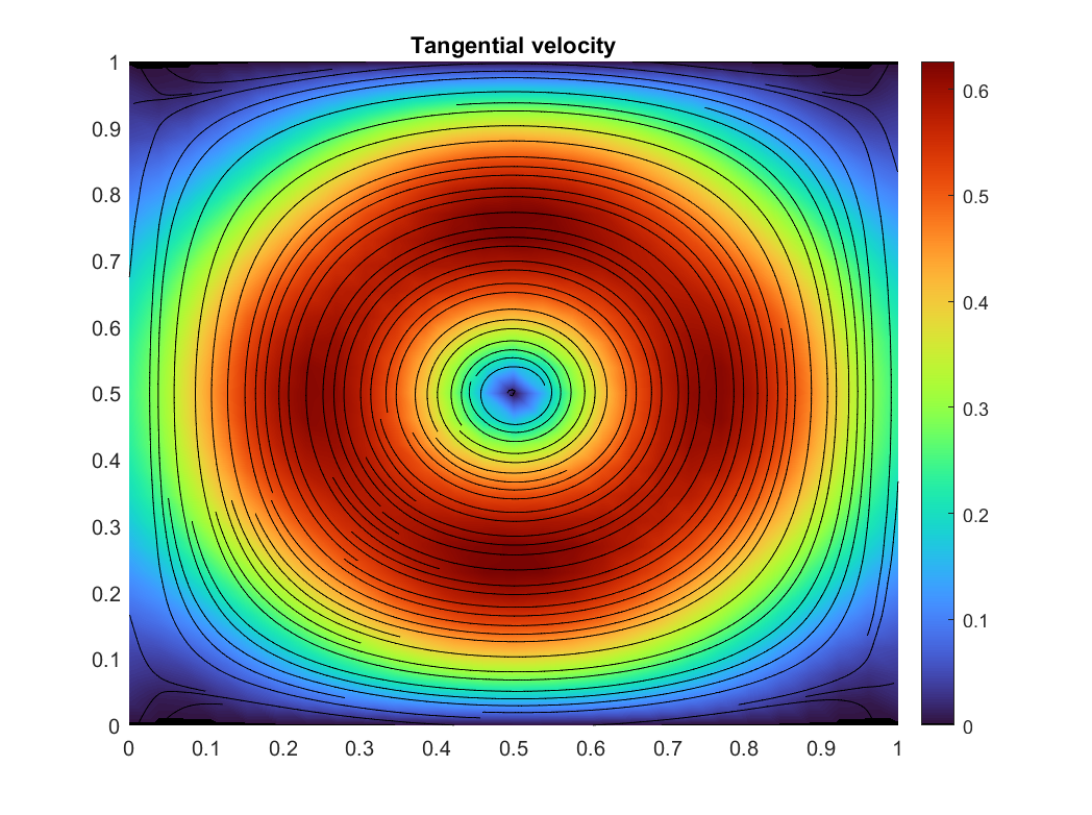}\\
\captionof{figure}{$\vg=0$}
\label{fig5:3}
\end{minipage}
\begin{minipage}{.49\textwidth}
\centering
\includegraphics[width=\linewidth]{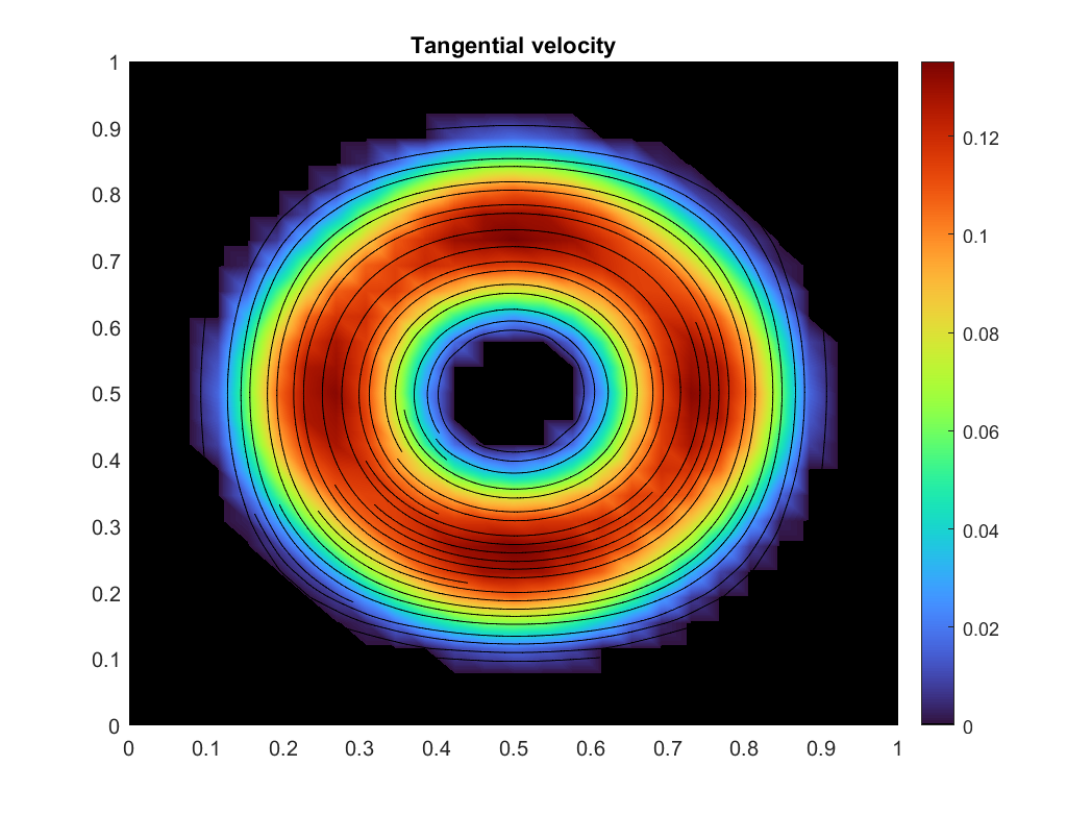}\\
\captionof{figure}{$\vg=5$}
\label{fig5:4}
\end{minipage}
\end{figure}

Results from each experiment are displayed in Table \ref{tab:ex1}, where $It$ is the~number of SSSN iterations, and $It_{in}$ is the sum of the GMRES iterations from all SSSN iterations. Table \ref{tab:ex1} demonstrates that the number of Newton iterations is mesh independent.

\begin{table}[ht!]	
\caption{Number of SSSN and GMRES iterations}
	\centering
	\begin{tabular}{r||rrr|rrr|rrr}
	\toprule
$\epsilon=1e^{-8}\color{black},\ \kappa=5$ & \multicolumn{3}{c|}{$g=0$} & \multicolumn{3}{c|}{$g=5$}& \multicolumn{3}{c}{$g=10$}\\ 
	$n_p$/$n_t$/$n_s$ & $It$ & $It_{in}$  &CPU& $It$ & $It_{in}$ & CPU& $It$ & $It_{in}$ &CPU\\
         \midrule
729/ 2560 / 63 & 4 & 70 & 0.11 & 6 & 83 & 0.08 & 4 & 72 & 0.09\\ 
1331/ 5000 / 99 & 5 & 101 & 0.23 & 7 & 124 & 0.53 & 4 & 81 & 0.16\\ 
2197/ 8640 / 143 & 5 & 104 & 0.42 & 6 & 120 & 0.38 & 4 & 75 & 0.25\\ 
3375/ 13720 / 195 & 5 & 104 & 0.89 & 7 & 168 & 1.17 & 4 & 94 & 0.70\\ 
6859/ 29160 / 323 & 5 & 115 & 3.36 & 6 & 136 & 3.58 & 4 & 92 & 2.45\\ 
9261/ 40000 / 399 & 5 & 126 & 5.30 & 7 & 216 & 7.59 & 3 & 68 & 3.20\\ 
12167/ 53240 / 483 & 5 & 133 & 7.38 & 8 & 265 & 12.88 & 4 & 97 & 5.45\\ 
15625/ 69120 / 575 & 5 & 141 & 11.36 & 7 & 255 & 17.30 & 4 & 99 & 8.97\\ 
19683/ 87880 / 675 & 5 & 148 & 18.44 & 8 & 263 & 25.72 & 4 & 102 & 18.69\\ 
\end{tabular}
\label{tab:ex1}
\end{table}
\newpage


\subsection*{Example 2}
Another example deals with the Stokes problem with stick-slip boundary conditions on the mesh of the Cerebral Aneurysm (see~\cite{aneurysm}) from the Vascular Model Repository [www.vascularmodel.com], which is an open-source database of cardiovascular models. This 3D domain $\Omega$ has a boundary $\partial\Omega$ divided into four non-empty disjoint parts such that $\partial \Omega = \overline{\gamma}_D \cup\overline{\gamma}_S\cup\overline{\gamma}_{N_1}\cup\overline{\gamma}_{N_2}$. 

On the set $\gamma_D$ one has a Dirichlet boundary condition where $\bu_D$ is described by a parabolic function which acts as an inlet. Zero Neumann "Do-nothing" condition is set on $\gamma_{N_1}$ and $\gamma_{N_2}$, which will provide an outlet, and impermeability and a stick-slip boundary condition are applied on $\gamma_S$. Kinematic viscosity is set to $\nu=0.9$ and the adhesive coefficient $\kappa=10$. The 2D scheme of the domain $\Omega$ is shown in figure \ref{fig:aneurysm_scheme}. The mesh has the following parameters: the number of nodes $n_p=13873$, the number of tetrahedra $n_t=73833$, and the number of nodes on $\gamma_S$ $n_s=4164$. 

The SSSN method has been tested for various values of the slip bound $g$.

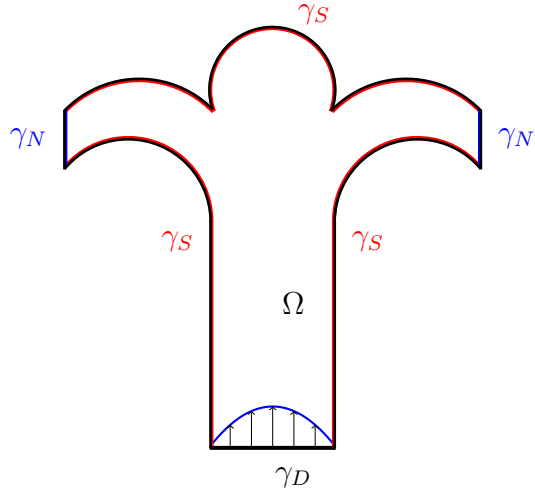
\begin{figure}[ht!]
\centering
\definecolor{ffffff}{rgb}{0.9,0.9,0.9}
\begin{tikzpicture}[x=5.5cm,y=5.5cm]
\draw[scale=1.0, domain=-0.15:0.15,smooth,variable=\x,blue,line width=0.8pt] plot({\x},{-4.2166*(\x)*(\x)+0.35});

 \draw[line width=1.5pt,color=red] (0.14,1.064) arc[start angle=-20,end angle=200,radius=0.15];
 \draw[line width=1pt,color=black] (0.142,1.064) arc[start angle=-20,end angle=198,radius=0.152];

\draw[line width=1.5pt,color=red] (-0.148,0.25) -- (-0.148,0.8);
\draw[line width=1.5pt,color=red] (0.148,0.25) -- (0.148,0.8);

 \draw[line width=1.5pt,color=red] (-0.503,0.928) arc[start angle=140,end angle=0,radius=0.201];
 \draw[line width=1.5pt,color=red] (-0.503,1.063) arc[start angle=136,end angle=44,radius=0.25];

 \draw[line width=1.5pt,color=red] (0.503,0.928) arc[start angle=40,end angle=180,radius=0.201];
 \draw[line width=1.5pt,color=red] (0.503,1.063) arc[start angle=44,end angle=136,radius=0.25];

 \draw[line width=1.5pt,color=blue] (0.501,0.93) -- (0.501,1.068);
 \draw[line width=1.5pt,color=blue] (-0.501,0.93) -- (-0.501,1.068);

 \draw[line width=1pt,color=black] (0.503,0.925) arc[start angle=40,end angle=180,radius=0.2];
 \draw[line width=1pt,color=black] (0.503,1.066) arc[start angle=44,end angle=136,radius=0.251];

 \draw[line width=1pt,color=black] (-0.503,0.925) arc[start angle=140,end angle=0,radius=0.2];
 \draw[line width=1pt,color=black] (-0.503,1.066) arc[start angle=136,end angle=45,radius=0.251];

 \draw[line width=1pt,color=black] (0.5032,0.9234) -- (0.5032,1.0675);
\draw[line width=1pt,color=black] (-0.5032,0.9234) -- (-0.5032,1.0675);

\draw[line width=1pt,color=black] (-0.15,0.25) -- (-0.15,0.8);
\draw[line width=1pt,color=black] (0.15,0.25) -- (0.15,0.8);

 \draw[line width=1.5pt,color=black] (-0.1525,0.25) -- (0.1525,0.25);

\draw[color=red] (0.1,1.3) node {$\gamma_S$};

\draw[color=blue] (-0.59,1.0) node {$\gamma_N$};
\draw[color=blue] (0.59,1.0) node {$\gamma_N$};

\draw[color=red] (-0.23,0.75) node {$\gamma_S$};
\draw[color=red] (0.23,0.75) node {$\gamma_S$};

\draw[color=black] (0.05,0.18) node {$\gamma_D$};
\draw[color=black] (0.05,0.6) node {$\Omega$};

\draw [->][ultra thin] (-0.1027,0.25) -- (-0.1027,0.3056);
\draw [->][ultra thin] (-0.0513,0.25) -- (-0.0513,0.3389);
\draw [->][ultra thin] (0,0.25) -- (0,0.35);
\draw [->][ultra thin] (0.0513,0.25) -- (0.0513,0.3389);
\draw [->][ultra thin] (0.1027,0.25) -- (0.1027,0.3056);

\end{tikzpicture}
  \caption{2D scheme of flow in Cerebral Aneurysm}
  \label{fig:aneurysm_scheme}
\end{figure}
 Figures \ref{fig5:5} show the absolute value of the velocity at the boundary of $\Omega$ for two values of $g=0$ and $g=10$.

\begin{figure}[ht!]
\includegraphics[width=1.0\linewidth]{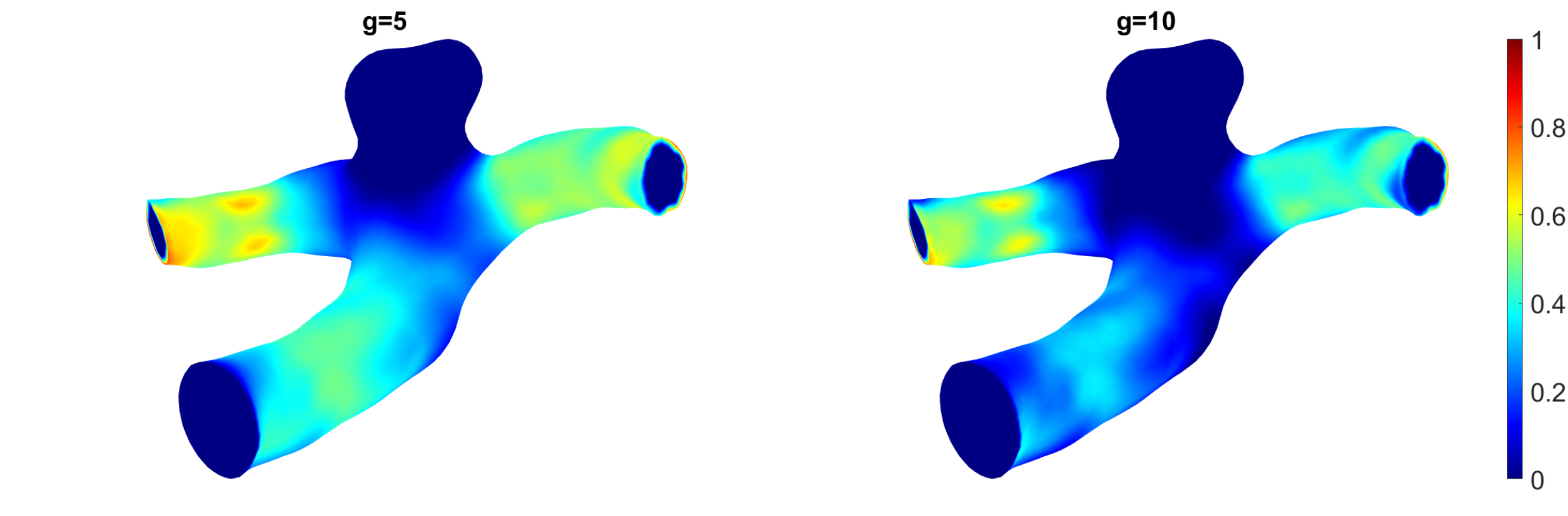}\\
\captionof{figure}{Tangent velocity}
\label{fig5:5}
\end{figure}

The results for all three variants are shown in Table \ref{tab:ex2}. Here, for each SSSN iteration $It$, the residue $Res$, the number of GMRES iterations $It_{in}$, and the number of nodes that are in the slip or stick state are shown. In the~SSSN method, the stopping criterion was set to $1e^{-8}$.

\begin{table}[ht!]
\caption{Number of SSNS and GMRES iterations for $\epsilon=1e^{-8}$}
	\centering
    \scriptsize
	\begin{tabular}{r||rrl|rrr|rrr}
	\toprule
$\kappa=10$ & \multicolumn{3}{c|}{$g=0$} & \multicolumn{3}{c|}{$g=5$}& \multicolumn{3}{c}{$g=10$}\\ 
	$It$ & $Res$ & $It_{in}$ & $Slip/Stick$ &  $Res$ & $It_{in}$ & $Slip/Stick$  &  $Res$ & $It_{in}$ & $Slip/Stick$  \\
         \midrule
 1 & 1.41e+00 & 79 & 43/4121 & 1.42e+00 & 87 & 33/4131 & 1.39e+00 & 69 & 21/4143 \\
2 & 1.37e+00 & 119 & 4164/0 & 1.22e+00 & 130 & 2886/1278 &  9.80e-01 & 154 & 1431/2733 \\
3 & 1.17e-02 & 118 & 4164/0 &  9.26e-01 & 116 & 3351/813 & 8.29e-01 & 131 & 2501/1663 \\
4 & 2.63e-04 & 130 & 4164/0 &  1.56e-01 & 105 & 3429/735 &  4.66e-01 & 121 & 2926/1238 \\
5 & 3.37e-06 & 140 & 4164/0 & 2.28e-02 & 82 & 3442/722 &  1.24e-01 & 119 & 3027/1137 \\
6 & 2.90e-08 & 144 & 4164/0 & 1.36e-03 & 104 & 3448/716 & 2.23e-02 & 120 & 3048/1116 \\
7 & 2.06e-10 & 146 & 4164/0 & 2.46e-05 & 157 & 3448/716 & 1.33e-03 & 93 & 3048/1116 \\
8 &  \text{Converg} &  &  & 1.05e-06 & 170 & 3449/715 &  1.72e-05 & 187 & 3048/1116 \\
9 &  &  &  &  1.68e-07 & 60 & 3449/715 & 2.28e-06 & 182 & 3048/1116 \\
10 & & & & 1.86e-07 & 86 & 3449/715 &  4.04e-08 & 204 & 3048/1116 \\
11 & && &  2.05e-07 & 14 & 3449/715 & 1.12e-09 & 203 & 3048/1116 \\
12 & && & 2.28e-07 & 35 & 3449/715 &  \text{Converg}\\
13 & && & 2.53e-07 & 141 & 3449/715 &  \\
14 & && & 2.74e-07 & 174 & 3449/715 &  \\
15 & && & 1.79e-07 & 230 & 3449/715 & \\
16 & && &  1.02e-07 & 202 & 3449/715 & \\
17 & && &  4.08e-08 & 257 & 3449/715 & \\
18 & && & 1.38e-08 & 264 & 3449/715 & \\
19 & && & 3.66e-09 & 276 & 3449/715 & \\
20 & & & & \text{Converg} &  & & \\
\end{tabular}
\label{tab:ex2}
\end{table}
\newpage

\subsection*{Conclusion}
We proposed an algorithm for the numerical solution of the 3D Stokes problem with Navier-Tresca stick-slip boundary conditions based on the recently developed SCD (Subspace Containing Derivative) semismooth* Newton method (SSSN method). This algorithm can also be used to solve the~Navier-Stokes problem and the unsteady Navier-Stokes problem after a relatively simple modification of the implementation.

The suggested method has been compared by the authors with the semismooth Newton method presented in~\cite{haslinger-kucera-sassi-satek}. We used this method for the same examples presented in the last section of the paper. Based on numerical experiments, we report that the SSSN method yields slightly better convergence rates than the semismooth Newton method (\cite{haslinger-kucera-sassi-satek}). Further, this method does not require to cast the problem in form of an equation.

\subsection*{Acknowledgement}
This research has been conducted with the financial support\linebreak of the European Union under the REFRESH – Research Excellence\linebreak for Region Sustainability and High-tech Industries project number
\linebreak CZ.10.03.01/00/22\_003/0000048 through the Operational Programme Just Transition.

\bibliography{Semismooth_Newton_star}

\end{document}